\documentclass[11pt]{amsart}
\usepackage{mystyle}
\usepackage[foot]{amsaddr}

\newcommand{\tD}{\widetilde{D}}

\title{A note on Hamilton decompositions of even-regular multigraphs}

\author[Pfenninger]{Vincent Pfenninger}
\email{pfenninger@math.tugraz.at}
\address{Institute of Discrete Mathematics, Graz University of Technology}

\thanks{ 
This project has received partial funding from the European Research
Council (ERC) under the European Union's Horizon 2020 research and innovation programme (grant agreement no.\ 786198).
}

\begin{document}

\begin{abstract}
    In this note, we prove that every even regular multigraph on $n$ vertices with multiplicity at most $r$ and minimum degree at least $rn/2 + o(n)$ has a Hamilton decomposition. This generalises a result of Vaughan who proved an asymptotic version of the multigraph $1$-factorisation conjecture. We derive our result by proving a more general result which states that dense regular multidigraphs that are robust outexpanders have a Hamilton decomposition. This in turn is derived from the corresponding result of K\"uhn and Osthus about simple digraphs.
\end{abstract}

\maketitle

\section{Introduction}
A \emph{multidigraph} (or \emph{directed multigraph}) $D$ is a pair $(V(D), E(D))$ where $V(D)$ is a finite set and~$E(D)$ is a multiset with elements from the set $(V(D)\times V(D))\setminus \{(v,v)\colon v \in V(D)\}$ (so loops are not allowed).
We call $V(D)$ the \emph{set of vertices of $D$} and $E(D)$ the \emph{multiset of edges of $D$}. 
The \emph{multiplicity of a multidigraph $D$} is the maximum number of times an edge appears in $E(D)$.
For a vertex $v \in V(D)$, we denote by $d^+_D(v)$ the number of \emph{outedges} at $v$ in $D$, that is, edges of the form $(v, x)$ for some $x \in V(D)$ (counting multiplicities). Similarly, $d^-_D(v)$ is the number of \emph{inedges} at $v$ in $D$. We say that $D$ is $s$-regular, if $d^+_D(v) = d^-_D(v) = s$ for every $v \in V(D)$.
A \emph{Hamilton cycle} in a multidigraph is a directed cycle covering all the vertices. A \emph{Hamilton decomposition} of a multidigraph $D$ is a partition of $E(D)$ such that every part is the edge set of a Hamilton cycle of~$D$. We also use analogous definitions for multigraphs (the undirected analogue to multidigraphs).

Vaughan proved that if $n$ is even and~$G$ is a regular multigraph on $n$ vertices with multiplicity at most $r$ and minimum degree at least $rn/2 + o(n)$, then~$G$ has a $1$-factorisation\footnote{A \emph{$1$-factorisation} is a decomposition of the edge set into perfect matchings.} \cite{Vaughan2013}. This is an approximate version of the multigraph $1$-factorisation conjecture of Plantholt and Tipnis \cite{Plantholt2001} which is a generalisation of the $1$-factorisation conjecture \cite{Chetwynd1985}. The $1$-factorisation conjecture states that if~$G$ is an $s$-regular graph on $n$ vertices where $n$ is an even integer and $s \geq n/2$, then~$G$ has a $1$-factorisation. The multigraph $1$-factorisation conjecture states that if~$G$ is an $s$-regular multigraph on $n$ vertices with multiplicity at most $r$ where $n$ is an even integer and $s \geq rn/2$, then~$G$ has a $1$-factorisation. The $1$-factorisation conjecture was proved for all large graphs by Csaba, K\"uhn, Lo, Osthus, and Treglown \cite{Csaba2016}. Our aim in this note is to generalise the result of Vaughan by showing that if in addition the multigraph is even-regular, then it has a Hamilton decomposition.

\begin{theorem} \label{HCD_thm}
    For every $\eps \in (0,1)$ and $r \in \mathbb{N}$, there exists $n_0 \in \mathbb{N}$ such that the following holds for all $n \geq n_0$. Let~$G$ be an $s$-regular $n$-vertex multigraph with multiplicity at most $r$ where $s \in 2\mathbb{N}$ with $s \geq rn/2 + \eps n$. Then~$G$ has a Hamilton decomposition. 
\end{theorem}

Observe that this does indeed imply the theorem of Vaughan since if~$G$ is an odd-regular graph on an even number of vertices we can first remove a perfect matching and then apply our result to the remaining graph.

We also prove a directed analogue of \cref{HCD_thm}. 

\begin{theorem} \label{HCD_thm_directed}
    For every $\eps \in (0,1)$ and $r \in \mathbb{N}$, there exists $n_0 \in \mathbb{N}$ such that the following holds for all $n \geq n_0$. Let $D$ be an $s$-regular $n$-vertex multidigraph with multiplicity at most $r$ where $s \geq rn/2 + \eps n$. Then $D$ has a Hamilton decomposition. 
\end{theorem}

We derive these results from a more general result about multidigraphs that are robust outexpanders. To state this result we need the following definitions. 

For an $n$-vertex digraph $D$, a set $S \subseteq V(D)$, and $\nu \in (0,1)$, we define the \emph{$\nu$-robust outneighbourhood of $S$ in $D$} to be $RN_{\nu, D}^+(S) \coloneqq \{ v \in V(D) \colon |S \cap N_D^-(v)| \geq \nu n\}$.
For $\nu, \tau \in (0,1)$, a \emph{robust $(\nu, \tau)$-outexpander} is a simple\footnote{A multidigraph is \emph{simple} if its multiplicity is $1$ (such a multidigraph is also called a \emph{digraph} or a \emph{simple digraph}).} $n$-vertex digraph $D$ such that for each set $S \subseteq V(D)$ with $\tau n \leq |S| \leq (1-\tau) n$, we have $|RN_{\nu, D}^+(S)| \geq |S| + \nu n$. 
For a multidigraph $D$, the \emph{underlying simple digraph of $D$} is the digraph $\tD$ obtained from $D$ by dropping any multiple edges (that is an ordered pair of vertices is an edge in $\tD$ if and only if it appears (at least once) in $E(D)$). We analogously define the \emph{underlying simple graph} of a multigraph.

The \emph{hierarchy} notation $0 < a \ll b < 1$ is a short form of stating that there exists a nondecreasing function $a_0 \colon (0,1) \rightarrow (0,1)$ such that the statement that follows after holds for all $a,b \in (0,1)$ with $a \leq a_0(b)$ (that is the statement holds as long as $a$ is sufficiently small in terms of $b$). Hierarchies with more variables are defined similarly and whenever $1/a$ appears in a hierarchy we implicitly assume that $a$ is a positive integer.

The following theorem is our main result from which we derive all other results in this note.

\begin{theorem} \label{HCD_thm2}
    Let $1/n \ll \nu \ll \tau \ll 1/r, \alpha \leq 1$. Let $D$ be a $n$-vertex multidigraph with multiplicity at most $r$ such that the following hold.
    \begin{enumerate}[label = \upshape{(\alph*)}]
        \item $D$ is $s$-regular for some $s \geq \alpha n$.
        \item The underlying simple digraph of $D$ is a robust $(\nu, \tau)$-outexpander.
    \end{enumerate}
   Then $D$ has a Hamilton decomposition.
\end{theorem}

We also prove an analogous result for (undirected) multigraphs. For this we need the corresponding undirected version of the above definition of robust outexpansion.
For an $n$-vertex graph~$G$, a set $S \subseteq V(G)$, and $\nu \in (0,1)$, we define the \emph{$\nu$-robust neighbourhood of $S$ in~$G$} to be $RN_{\nu, G}(S) \coloneqq \{v \in V(G) \colon |S \cap N_G(v)| \geq \nu n\}$. 
For $\nu, \tau \in (0,1)$, a \emph{robust $(\nu, \tau)$-expander} is a simple $n$-vertex (undirected) graph~$G$ such that for each set $S \subseteq V(G)$ with $\tau n \leq |S| \leq (1-\tau) n$, we have $|RN_{\nu, G}(S)| \geq |S| + \nu n$. As discussed in \cite{Kuhn2014}, robust expansion is a very weak notion of quasirandomness, and thus, for example, random graphs of constant density are robust expanders with high probability.

\begin{theorem} \label{HCD_thm3}
    Let $1/n \ll \nu \ll \tau \ll 1/r, \alpha \leq 1$. Let~$G$ be a $n$-vertex multigraph with multiplicity at most $r$ such that the following hold.
    \begin{enumerate}[label = \upshape{(\alph*)}]
        \item~$G$ is $s$-regular for some $s \in 2\mathbb{N}$ with $s \geq \alpha n$.
        \item The underlying simple graph of~$G$ is a robust $(\nu, \tau)$-expander.
    \end{enumerate}
   Then~$G$ has a Hamilton decomposition.
\end{theorem}

The case $r=1$ of all these results was proved by K\"uhn and Osthus in \cite{KUHN2013}.
Throughout, we ignore floors and ceilings whenever doing so does not affect the argument. 

\section{Proofs of secondary results}

In this section, we prove that \cref{HCD_thm2} implies \cref{HCD_thm3,HCD_thm,HCD_thm_directed}.

\subsection{Proof of Theorem~\ref{HCD_thm3}}

To derive \cref{HCD_thm3} from \cref{HCD_thm2} we need the following two lemmas from \cite{Kuhn2014}. The first of these lemmas allows us to orient the edges of a robust expander in such a way that the resulting digraph is a robust outexpander and the in-degree and out-degree of every vertex are not too far apart. 

\begin{lemma}[{\cite[Lemma 3.1]{Kuhn2014}}] \label{lem:orient}
    Suppose that $1/n \ll \eta \ll \nu, \tau, \alpha < 1$. Suppose that~$G$ is a robust $(\nu, \tau)$-expander on $n$ vertices with $\delta(G) \geq \alpha n$. Then one can orient the edges of~$G$ in such a way that the resulting oriented graph $D$ satisfies the following\footnote{For positive reals $a,b,c,d$, we write $a = b(c \pm d)$ for $b(c-d) \leq a \leq b(c+d)$.}:
    \begin{enumerate}[label = \upshape{(\roman*)}]
        \item $D$ is a robust $(\nu/4, \tau)$-outexpander. 
        \item $d_{D}^+(x) = (1 \pm \eta) \frac{d_G(x)}{2}$ and $d_{D}^-(x) = (1 \pm \eta) \frac{d_G(x)}{2}$.
    \end{enumerate}
\end{lemma}

For a (simple) digraph $D$, we define its \emph{minimum semi-degree $\delta^0(D)$} to be
\[
    \delta^0(D) \coloneqq \min_{\substack{\circ \in \{+, -\} \\ v \in V(D)}} d^\circ_D(v).
\] The second lemma allows us to find a regular spanning subdigraph of a robust outexpander with linear minimum semi-degree that is still a robust outexpander.

\begin{lemma}[{\cite[Lemma 3.4]{Kuhn2014}}] \label{lem:factor}
    Suppose that $1/n \ll \nu' \ll \xi \ll \nu \leq \tau \ll \alpha < 1$. 
    Let $D$ be a robust $(\nu, \tau)$-outexpander on $n$ vertices with $\delta^0(D) \geq \alpha n$. Then $D$ contains a $\xi n$-factor\footnote{A $\xi n$-factor is a $\xi n$-regular spanning subdigraph.} which is still a robust $(\nu', \tau)$-outexpander.
\end{lemma}

We now show that \cref{HCD_thm2} implies its undirected analogue \cref{HCD_thm3}.

\begin{proof}[Proof of \cref{HCD_thm3} assuming \cref{HCD_thm2}]
    Choose new constants $\nu'$ and $\xi$ such that $1/n \ll \nu' \ll \xi \ll \nu \ll \tau \ll 1/r, \alpha$. 
    Note that we can write $G = G' \dot\cup H$ where $G'$ is the underlying simple graph of~$G$ and $H$ is a multigraph that is edge-disjoint from $G'$. By assumption, $G'$ is a robust $(\nu, \tau)$-expander with $\delta(G') \geq \alpha n / r$. By \cref{lem:orient}, there exists an orientation $D$ of $G'$ such that $D$ is a robust $(\nu/4, \tau)$-outexpander with $\delta^0(D) \geq \alpha n / 3r$. By \cref{lem:factor}, $D$ contains a $\xi n$-factor $F$ that is a robust $(\nu', \tau)$-outexpander. 
    Let $\tG$ be the undirected multigraph obtained from~$G$ by deleting the edges contained in $F$. Note that $\tG$ is $(s - 2\xi n)$-regular. Since $s$ is even, there exists a decomposition of $\tG$ into cycles (we allow 2-cycles, that is two parallel edges). Orienting each of these cycles consistently and then combining this with $F$ gives an orientation $\tD$ of~$G$ that is $s/2$-regular and such that its underlying simple digraph is a robust $(\nu', \tau)$-outexpander. By applying \cref{HCD_thm2} to $\tD$, we get a Hamilton decomposition of $\tD$ which, by dropping the orientations, gives the desired Hamilton decomposition of~$G$.
\end{proof}

\subsection{Proof of Theorem~\ref{HCD_thm_directed} and Theorem~\ref{HCD_thm}}

\cref{HCD_thm_directed} follows from \cref{HCD_thm2} by the following lemma which states that the minimum semi-degree condition on $D$ implies that $D$ is a robust outexpander.

\begin{lemma}[{\cite[Lemma 3.7]{Kuhn2014}}] \label{lem:min_deg_robust_directed}
    Let $1/n \ll \nu \ll \tau \ll \eps < 1$. Let $D$ be a digraph on $n$ vertices with minimum semi-degree $\delta^0(D) \geq (1/2 + \eps)n$. Then $D$ is a robust $(\nu, \tau)$-outexpander.
\end{lemma}

\begin{proof}[Proof of \cref{HCD_thm_directed} assuming \cref{HCD_thm2}]
    Let $\nu$ and $\tau$ be new constants such that $1/n \ll \nu \ll \tau \ll \eps, 1/r$. Note that the underlying simple digraph $D'$ of $D$ has minimum semi-degree $\delta^0(D') \geq (1/2 + \eps/r)n$. By \cref{lem:min_deg_robust_directed}, $D'$ is a robust $(\nu, \tau)$-outexpander. Hence we are done by \cref{HCD_thm_directed}.
\end{proof}

Analogously, \cref{HCD_thm} follows from \cref{HCD_thm3} by the following lemma (the undirected version of \cref{lem:min_deg_robust_directed}).

\begin{lemma}[{\cite[Lemma 3.8]{Kuhn2014}}] \label{lem:min_deg_robust}
    Let $1/n \ll \nu \ll \tau \ll \eps < 1$. Let~$G$ be a graph on $n$ vertices with minimum degree $\delta(G) \geq (1/2 + \eps)n$. Then~$G$ is a robust $(\nu, \tau)$-expander.
\end{lemma}

\begin{proof}[Proof of \cref{HCD_thm} assuming \cref{HCD_thm3}]
    Let $\nu$ and $\tau$ be new constants such that $1/n \ll \nu \ll \tau \ll \eps, 1/r$. Note that the underlying simple graph $G'$ of~$G$ has minimum degree $\delta(G') \geq (1/2 + \eps/r)n$. By \cref{lem:min_deg_robust}, $G'$ is a robust $(\nu, \tau)$-expander. Hence we are done by \cref{HCD_thm3}.
\end{proof}

\section{Proof of Theorem~\ref{HCD_thm2}}

In this section we prove \cref{HCD_thm2}, the main theorem of this note. The main ingredients for our proof are the following two results from \cite{KUHN2013} and \cite{Osthus2013}. They state, respectively, that a robust outexpander with linear minimum semi-degree has a Hamilton decomposition if it is regular and can be almost decomposed into Hamilton cycles if it is almost regular.

\begin{theorem}[{\cite[Theorem 1.2]{KUHN2013}}] \label{thm:regular_HCD}
    Let $1/n \ll \nu \ll \tau \ll \alpha \leq 1$. Let $D$ be a $s$-regular simple digraph on $n$ vertices that is a robust $(\nu, \tau)$-outexpander where $s \geq \alpha n$. Then $D$ has a Hamilton decomposition.
\end{theorem}

\begin{theorem}[{\cite[Corollary 1.2]{Osthus2013}}] \label{cor:almost_reg_HCD}
    Let $1/n \ll \xi \ll \delta \ll \nu \ll \tau \ll \alpha \leq 1$. Let $D$ be an $n$-vertex digraph such that
    \begin{enumerate}[label = \upshape{(\roman*)}]
        \item $d_{D}^+(v) = (\alpha \pm \xi)n$, $d_{D}^-(v) = (\alpha \pm \xi)n$ for all $v \in V(D)$ and
        \item $D$ is a robust $(\nu, \tau)$-outexpander.
    \end{enumerate}
    Then $D$ contains at least $(\alpha - \delta) n$ edge-disjoint Hamilton cycles.
\end{theorem}

To prove \cref{HCD_thm2}, also need the following four facts about robust outexpanders.
The first is the fact that robust outexpanders with linear minimum semi-degree are Hamilton-connected.

\begin{corollary}[{\cite[Corollary 6.9]{Kuhn2015}}] \label{cor:Hamilton_connected}
    Let $1/n \ll \nu \ll \tau \ll \alpha \leq 1$. Let $D$ be a simple digraph on $n$ vertices that is a robust $(\nu, \tau)$-outexpander with $\delta^0(D) \geq \alpha n$. Then for any $x, y \in V(D)$, $D$ contains a Hamilton path from $x$ to $y$.
\end{corollary}

The next fact is that a random subgraph of a robust outexpander is very likely to still be a robust outexpander.

\begin{lemma}[{\cite[Lemma 4.12]{Girao2023}}, {\cite[Lemma 3.2(ii)]{Kuhn2014}}] \label{lem:still_robust}
    Let $1/n \ll \nu \ll \tau, p \leq 1$. Let $D$ be a robust $(\nu, \tau)$-outexpander on $n$ vertices. Suppose that $\Gamma$ is the spanning random subgraph of $D$ obtained by taking each edge independently with probability $p$. Then, with probability at least $1- \exp(- \nu^3n^2)$, $\Gamma$ is a robust $(p\nu/2, \tau)$-outexpander. 
\end{lemma}

Finally, we need the following two simple facts which follow easily form the definition of robust outexpanders.

\begin{lemma}[{\cite[Lemma 4.2]{Girao2023}}] \label{lem:delete}
    Let $1/n \ll \eps \ll \nu \ll \tau \leq 1$ and let $D$ be digraph that is a robust $(\nu, \tau)$-outexpander on $n$ vertices. Then any digraph obtained from $D$ by deleting at most $\eps n$ inedges and at most $\eps n$ outedges at each vertex as well as deleting at most $\eps n$ vertices is a robust $(\nu - 2\eps, 2 \tau)$-outexpander.
\end{lemma}

\begin{lemma}[{\cite[Lemma 6.6]{Kuhn2015}}] \label{lem:connect}
    Let $1/n \ll \nu \ll \tau \leq 1$. Let $D$ be a simple digraph on $n$ vertices that is a robust $(\nu, \tau)$-outexpander with $\delta^0(D) \geq 2\tau n$. Then for any $x, y \in V(D)$, there exists a (directed) path in~$D$ from $x$ to $y$ of length at most $1/\nu$.
\end{lemma}

We are now ready to prove \cref{HCD_thm2}. The main idea of the proof is to randomly split the multidigraph $D$ into $r$ simple digraphs $D_1, \dots, D_r$. We then cover most edges in $D_2, \dots, D_r$ with Hamilton cycles using \cref{cor:almost_reg_HCD}. Using a small number of edges of $D_1$, we then find Hamilton cycles covering the remaining edges of $D_2, \dots, D_r$. Finally, we show that the remaining subgraph of $D_1$ has a Hamilton decomposition using \cref{thm:regular_HCD}. 

\begin{proof}[Proof of \cref{HCD_thm2}]
    Let $\alpha' \coloneqq s/n \geq \alpha$.
    Let $\tD$ be the underlying simple digraph corresponding to $D$. For each $e \in E(\tD)$, we denote by $m_e$ the multiplicity with which $e$ appears in $D$. For each $e \in E(\tD)$ independently, we let $X_e$ be a subset of $[r]$ of size $m_e$ chosen uniformly at random. For each $i \in [r]$, we let $D_i$ be the digraph with $V(D_i) = V(D)$ and $E(D_i) = \{e \in E(\tD) \colon i \in X_e\}$. Note that $D_i$ is a simple digraph for each $i \in [r]$ and that $\dot\bigcup_{i \in [r]} D_i = D$ as a multidigraph. 

    Let $\xi$ and $\delta$ be new constants with $1/n \ll \xi \ll \delta \ll \nu$. Let $u \in V(\tD)$ and $i \in [r]$. We have 
    \[
        \mathbb{E}[d_{D_i}^+(u)] = \sum_{\substack{v \in V(D)\\ e = (u,v) \in E(\tD)}} \binom{r-1}{m_e-1}/\binom{r}{m_e} = \sum_{\substack{v \in V(D)\\ e = (u,v) \in E(\tD)}} m_e/r = d_D^+(u)/r = s/r.
    \]
    And similarly, $\mathbb{E}[d_{D_i}^-(u)] = s/r$. Thus by the Chernoff bound, we have 
    \[
        \mathbb{P}[d_{D_i}^+(v) \neq (1 \pm \xi)s/r], \:\mathbb{P}[d_{D_i}^-(v) \neq (1 \pm \xi)s/r] \leq \exp(-\Omega(n)).
    \]
    Note that there is a coupling of $D_i$ and the digraph $\tD_{1/r}$ obtained from $\tD$ by taking each edge independently with probability $1/r$ such that $\tD_{1/r} \subseteq D_i$.
    Hence, since $\tD$ is a robust $(\nu, \tau)$-outexpander, we have by \cref{lem:still_robust}, with probability at least $1-\exp(\nu^3 n^2)$, $D_i$ is a robust $(\nu /2r, \tau)$-outexpander. 
    Hence by a union bound, we have that with probability $1-o(1)$ for every $i \in [r]$,
    \begin{enumerate}[label = \upshape{(\roman*)}]
        \item $d_{D_i}^+(v) = (1 \pm \xi)s/r$, $d_{D_i}^-(v) = (1 \pm \xi)s/r$ for all $v \in V(D)$ and
        \item $D_i$ is a robust $(\nu /2r, \tau)$-outexpander.
    \end{enumerate}
    Fix such a choice of $(D_1, \dots, D_r)$.
    Note that $(\alpha' / r - \xi \alpha' /r)n = (1- \xi)s/r \leq (1+\xi)s/r = (\alpha' / r +  \xi \alpha' /r)n$. Hence by \cref{cor:almost_reg_HCD}, for each $i \in [r]$, there is a set $\cH_i$ of at least $(\alpha'/r - \delta)n$ edge-disjoint Hamilton cycles in $D_i$. For each $i \in [r]$, let $D_i' = D_i - \bigcup_{H \in \cH_i} E(H)$ and note that $\Delta^{\pm}(D_i') \coloneqq \max_{\circ \in \{+,-\}, v \in V(D_i')} d_{D_i'}^\circ(v) \leq 2 \delta n -1$. 
    
    Let $D' = D_1 \cup \bigcup_{2 \leq i \leq r} D_i'$. Observe that it suffices to show that $D'$ has a Hamilton decomposition. Note that $D'$ is $s'$-regular for some $s' \geq (1-2\xi)\alpha' n/r$ and since $D_1$ is a robust $(\nu /2r, \tau)$-outexpander, $D'$ is also a robust $(\nu /2r, \tau)$-outexpander.

    We now show that, for each $i \in [r]$, there exists a decomposition of $D_i'$ into a set $\cM_i$ of at most $16 \delta^{1/2} n$ matchings each of size at most $\delta^{1/2}n$. Indeed, observe that, for each $i \in [r]$, by decomposing $D_i'$ into two oriented graphs\footnote{An \emph{oriented graph} is a simple digraph such that for any vertices $x$ and $y$ at most one of $(x,y)$ and $(y,x)$ is an edge.} and then applying Vizing's theorem to the underlying graphs, we obtain a decomposition of the edges of $D_i'$ into a set of at most $8\delta n$ matchings. From this the desired set of matchings is obtained by splitting each matching of size larger than $\delta^{1/2}n$ into $2 / \delta^{1/2}$ matchings of as equal size as possible.

    Let $\cM = \bigcup_{2 \leq i \leq r} \cM_i = \{M_1, \dots, M_t\}$, where $t \leq 16r\delta^{1/2}n \leq \delta^{1/4}n$. We construct a set of edge-disjoint paths $\cP = \{P_1, \dots, P_t\}$ in $D'$ such that $M_i \subseteq E(P_i)$ and $|V(P_i)| \leq \delta^{1/4}n$ for each $i \in [t]$ as follows. 
    Let $i \in [t]$ and suppose that $P_1, \dots, P_{i-1}$ have already been constructed. Let $M_i = \{(u_j, v_j) \colon j \in [m]\}$ where $m \leq \delta^{1/2}n$ and let $D^* = D' - \bigcup_{1 \leq j \leq i-1}E(P_j) - \bigcup_{i+1 \leq j \leq t} M_i$. 
In order to show that $D^*$ contains a path $P_i$ such that $M_i \subseteq E(P_i)$ and $|V(P_i)| \leq \delta^{1/4}n$, we show by induction on $j \in [m]$ that there exists a path $Q_j$ from $u_1$ to $v_j$ in $D^* - \bigcup_{j+1 \leq j' \leq m} \{u_j', v_j'\}$ of length at most $2 j \nu^{-2}$ such that $Q_j$ contains the edges $(u_{j'}, v_{j'})$ for all $j' \in [j]$ (then $Q_m$ is the desired path $P_i$). 
    To that end, let $j \in [m]$ and suppose that $Q_{j-1}$ is a path from $u_1$ to $v_{j-1}$ in $D^* - \bigcup_{j \leq j' \leq m} \{u_j', v_j'\}$ of length at most $2 (j-1) \nu^{-2}$ such that $Q_{j-1}$ contains the edges $(u_{j'}, v_{j'})$ for all $j' \in [j-1]$. Let $D^{**} = D^* - \{v_j\} - \bigcup_{j+1 \leq j' \leq m} \{u_{j'}, v_{j'}\} - (V(Q_{j-1}) \setminus \{v_{j-1}\})$.
    Observe that $D^{**}$ is obtained from $D'$ by deleting at most $\delta^{1/4}n$ inedges and at most $\delta^{1/4}n$ outedges at every vertex and then deleting at most $2 \delta^{1/2}n + 2(j-1)\nu^{-2} \leq \delta^{1/4}n$ vertices. Hence by \cref{lem:delete}, $D^{**}$ is a robust $(\nu/2r - 3 \delta^{1/4}, 2\tau)$-outexpander. Since $\delta^0(D^{**}) \geq 4\tau n$, by \cref{lem:connect}, we have that there exists a path $P^*$ in $D^{**}$ from $v_{j-1}$ to $u_j$ of length at most $\nu^{-2}$. Prepending $Q_{j-1}$ and appending the edge $(u_j, v_j)$ to $P^*$ gives the desired path $Q_j$.

    We now construct a set $\cH = \{H_1, \dots, H_t\}$ of edge-disjoint Hamilton cycles in $D'$ such that, for each $i \in [t]$, $E(P_i) \subseteq E(H_i)$. Let $i \in [t]$ and suppose that $H_1, \dots, H_{i-1}$ have already been constructed. Let $x$ and $y$ be the first and last vertex of $P_i$, respectively. Let $D^\diamond$ be the digraph obtained from $D'$ by deleting the edges in $\bigcup_{1\leq j \leq i-1} E(H_j) \cup \bigcup_{i+1 \leq j \leq t} E(P_j)$ and deleting all vertices of $P_i$ except $x$ and $y$. Observe that $D^\diamond$ is obtained from $D'$ by deleting at most $\delta^{1/4}n$ inedges and at most $\delta^{1/4}n$ outedges at every vertex and then deleting at most $\delta^{1/4}n$ vertices. Hence by \cref{lem:delete}, $D^\diamond$ is a robust $(\nu/2r - 2\delta^{1/4}, 2 \tau)$-outexpander. Since $\delta^0(D^\diamond) \geq \alpha' n /2r$, by \cref{cor:Hamilton_connected}, $D^\diamond$ contains a Hamilton path from $y$ to $x$ which together with $P_i$ forms the desired Hamilton cycle $H_i$ in $D'$. 

    Let $D'' = D' - \bigcup_{H \in \cH} E(H)$. It now suffices to show that $D''$ has a Hamilton decomposition. Note that $D''$ is a simple digraph as $D'' \subseteq D_1$. Since $D''$ is obtained from $D'$ by deleting at most~$\delta^{1/4} n$ inedges and $\delta^{1/4} n$ outedges at every vertex, we have that $D''$ is a robust $(\nu/2r - 2\delta^{1/4}, 2\tau)$-outexpander.
    Moreover, since $D'$ is $s'$-regular for some $s' \geq (1-2\xi)\alpha' n/r$, $D''$ is $s''$-regular for some $s'' \geq \alpha' n/2r$. Finally, by \cref{thm:regular_HCD}, $D''$ has a Hamilton decomposition.
\end{proof}

\section*{Acknowledgements} The author thanks Daniela K\"uhn and Deryk Osthus for helpful discussions.

\end{document}